\documentclass[12pt]{article}
\usepackage{amsmath}
\usepackage{amsthm}
\usepackage{amssymb}
\usepackage[T1]{fontenc}
\usepackage[sc]{mathpazo}
\usepackage{enumerate}
\usepackage{etoolbox}

\patchcmd{\section}{\scshape}{\bfseries}{}{}

\newcommand{\TV}{\mathrm{TV}}
\newcommand{\KL}{\mathrm{KL}}

\newcommand{\R}{\mathbb{R}}
\newcommand{\N}{\mathbb{N}}

\renewcommand{\subset}{\subseteq}

\newtheorem{theorem}{Theorem}[section]

\newtheorem{corollary}[theorem]{Corollary}

\theoremstyle{definition}

\theoremstyle{remark}

\newtheorem{example}[theorem]{Example}

\begin{document}
\title{A degree-linear upper bound for the P\'olya exponent of forms with fixed arity}
\author{Yulong Li \footnote{Department of Mathematics, University of Dayton. 300 College Park, Dayton, OH 45469, USA. \textbf{Email:} \texttt{yli004@udayton.edu}}
\and Colin Tan\footnote{Student at School of Humanities \& Social Sciences, Temasek Polytechnic, Singapore. 
\textbf{Subject class:} 
Primary 14Q20; Secondary 14P10, 26C05, 51M20, 60E15.
}}
\date{10 October 2026 (Version 1)}

\maketitle

\begin{abstract}
Let $P = \sum_{\omega} a_\omega x^\omega$ be a degree $d$ homogeneous polynomial in $x_0, \dots, x_n$ with real coefficients. 
P\'olya showed that, if $P$ is strictly positive on the standard $n$-simplex, then $(x_0 + \cdots + x_n)^M P$ has strictly positive coefficients for all $M \gg 0$. The least such $M$ is known as the \emph{P\'olya exponent} of $P$. de Klerk-Laurent-Parrilo showed that the P\'olya exponent is $O(d^2) P^{(0)}_{\max}/\lambda(P)$, where $P^{(0)}_{\max} := \max_\omega a_\omega/\binom{d}{\omega}$ and $\lambda(P)$ is the infimum of $P$ on the standard simplex. P\'olya's Positivstellensatz is a consequence of the classical approximation of the multivariate hypergeometric distribution by the multinomial distribution, as was already implicitly known to de Klerk-Laurent-Sun.
By applying known estimates of this approximation from probability theory, we obtain a different $O_n(d) L(P)/\lambda(P)$ upper bound of the P\'olya exponent, where $L(P) := \max |a_\omega|/\binom{d}{\omega}$. An example shows that $O_n(d)$ is tight.  
\end{abstract}


\section{P\'olya's Positivstellensatz} \label{sec: intro}

Bernstein showed that every univariate real polynomial $p \in \R[x]$ of degree $d$ that is strictly positive on the unit interval $[0, 1] \subset \R$ has strictly positive degree-$N$ Bernstein coefficients as long as $N \gg d$ \cite{Bernstein1915}. Here the degree-$N$ \emph{Bernstein coefficients} of $p = \sum_{j = 0}^N c_{j, N} b_{j, N}$ when written in terms of the Bernstein basis polynomials $b_{j, N} = \binom{N}{j} x^j(1-x)^{N - j}$ (for $j = 0, \dots, N$) are the $c_{j, N}$'s.

Bernstein's result may be reformulated projectively over the standard $1$-dimensional simplex $\Delta^1 \subset \R^2$, with vertices $(1, 0)$ and $(0, 1)$. 
Namely, if a binary form (i.e.\ homogeneous polynomial) $P = \sum_{j = 0}^d c_{j, d} \binom{d}{j} x^j y^{d - j}  \in \R[x, y]$ of degree $d$ is strictly positive on $\Delta^1$, then every coefficient of monomials in $(x + y)^{N - d} P$ of degree $N$ is strictly positive for all $N \gg d$.

P\'olya gave a generalisation to forms of arbitrary arity \cite{Polya1928}. Let $n \in \N := \{0, 1, 2, \dots \}$. The \emph{standard $n$-dimensional simplex} $\Delta^n \subset \R^{n + 1}$ has vertices $(1, 0, \dots, 0)$, $(0, 1, 0, \dots, 0)$, \dots, $(0, \dots, 0, 1) \in \R^{n + 1}$. 
Let ${\N^{n + 1}}_d \subset \N^{n + 1}$ denote the subset of $(n + 1)$-tuples $\beta = (\beta_0 \dots, \beta_n)$ of nonnegative integers satisfying $ \beta_0 + \dots + \beta_n = d$, where $d \in \N$. 
Say that a degree-$N$ real form $Q = \sum_{\beta \in {\N^{n + 1}}_N} a_\beta x^\beta$ in the indeterminates $x_0, \dots, x_n$ has \emph{strictly positive coefficients} if $a_\beta > 0$ for all $\beta \in {\N^{n + 1}}_N$. Here $x^\beta := {x_0}^{\beta_0} \cdots {x_n}^{\beta_n}$, whose coefficient in $Q$ we write as $a_\beta =: [x^\beta](Q)$.

\begin{theorem}[P\'olya] \label{thm: PolyaPositivstellensatz}
If an $(n + 1)$-ary form $P \in \R[x_0, \dots, x_n]$ of degree $d$ is strictly positive on $\Delta^n$, then $(x_0 + \cdots + x_n)^{N- d} P$ has strictly positive coefficients for all $N \gg d$. 
\end{theorem}

Theorem \ref{thm: PolyaPositivstellensatz} has become known as \emph{P\'olya's Positivstellensatz} as the representability of $P = Q/(x_0 + \cdots + x_n)^{N - d}$ as a rational function whose numerator $Q$ has strictly positive coefficients witnesses the strict positivity of $P$ on $\Delta^n$. P\'olya's original proof was analytic. 

\paragraph{A probabilistic proof.}

P\'olya's Positivstellensatz is in fact a consequence of the approximation of the multivariate hypergeometric distribution by the multinomial distribution. The idea of this probabilistic interpretation can already be found in de Klerk-Laurent-Sun's \cite{deKlerkLaurentSun2015}. 
For $\chi \in {\N^{n + 1}}_N$, the \emph{multivariate hypergeometric distribution} $H_{N, d, \chi}$ of sampling $d$ balls without replacement from an urn with $N$ coloured balls, where there are exactly $\chi_i$ balls of the $i$-th colour, has probability mass function $H_{N, d, \chi}(\omega) = \binom{d}{\omega} \binom{N - d}{\chi - \omega}/\binom{N}{\chi}$ (for $\omega \in {\N^{n + 1}}_d$). Here $\binom{N}{\chi} = N!/(\chi_0! \dots \chi_n!)$ is the multinomial coefficient as usual, for $\chi \in {\N^{n + 1}}_N$. From probability theory,
\begin{equation} \label{eq: probabilisticApprox}
H_{N, d, \chi} \sim B_{d, \chi/N} \text{ for $N \gg d$},
\end{equation}
where $B_{d, \chi/N} = B_{d, (\chi_0/N, \dots, \chi_n/N)}$ is the \emph{multinomial distribution} of sampling $d$ balls with replacement from the same urn. The probability mass function of $B_{d, x}$, where $x \in \Delta^n$, is given explicitly by $B_{d, x}(\omega) = \binom{d}{\omega}x^\omega$ (for $\omega \in {\N^{n + 1}}_d$). The nearness of the two distributions $H_{N, d, \chi}$ and $B_{d, \chi/N}$ may be quantified in various ways \cite{JohnsonGavalakisKontoyiannis2025} (see \S\ref{sec: probabilityTheory} below for more details). 
\eqref{eq: probabilisticApprox} leads to the following probabilistic proof of Theorem \ref{thm: PolyaPositivstellensatz}.
 
\begin{proof}[Proof of Theorem \ref{thm: PolyaPositivstellensatz}]
Suppose that $P \in \R[x_0, \dots, x_n]$ of degree $d$ is strictly positive on $\Delta^n$. Then for all $N \ge d$, for all $\chi \in {\N^{n + 1}}_N$,
\begin{align}
\frac{1}{\binom{N}{\chi}}[x^\chi]((x_0 + \dots + x_n)^{N - d} P) 
 &= \sum_{\omega \in {\N^{n + 1}}_d} \frac{1}{\binom{d}{\omega}} [x^\omega](P)  H_{N, d, \chi}(\omega) \notag \\
&\approx \sum_{\omega \in {\N^{n + 1}}_d} \frac{1}{\binom{d}{\omega}} [x^\omega](P) B_{d, \chi/N}(\omega) \label{eq: theApprox}\\
&= P(\chi/N) > 0,  \notag \end{align}
since $\chi/N \in \Delta^n$. The error of the approximation \eqref{eq: theApprox} is uniformly small in $\chi \in {\N^{n + 1}}_{N}$ if $N \gg d$.
\end{proof}

The reader should now expect that an appropriate hard uniform estimate of the approximation \eqref{eq: theApprox} should yield a corresponding estimate of the least $N \ge d$ such that $(x_0 + \dots + x_n)^{N - d} P$ has strictly positive coefficients. Following de Loera-Santos's definition of the \emph{P\'olya exponent} of a form $P \in \R[x_0, \dots, x_n]$ \cite{deLoeraSantos1996}, we denote
\begin{equation} \label{eq: PolyaExponent}
\pi(P) := \inf\{M \in \N : \, \text{$(x_0 + \cdots + x_n)^{M} P$ has strictly positive coefficients}. \}
\end{equation}
P\'olya's Positivstellensatz (Theorem \ref{thm: PolyaPositivstellensatz}) is the assertion that $\pi(P) < \infty$ whenever $P$ is strictly positive on $\Delta^n$. We shall bound $\pi(P)$ above using known quantitative estimates of \eqref{eq: probabilisticApprox} from probability theory.

\section{Estimates from probability theory} \label{sec: probabilityTheory}
The \emph{total variation distance} between two probability distributions $H$ and $B$ on a finite sample space $\Omega$ is $d_{\TV}(H, B) := \frac{1}{2}\sum_{\omega \in \Omega}|H(\omega) - B(\omega)|$. Pinsker's inequality, e.g.\ \cite[Lemma 2.5(i), p.\ 88]{Tsybakov2009}, bounds $d_{TV}$ above ---
\begin{equation} \label{eq: Pinsker}
d_{\TV}(H, B) \le \sqrt{\frac{1}{2} D_{\KL}(H \parallel B)},
\end{equation}
--- in terms of the \emph{relative entropy} $D_{\KL}(H \parallel B) := \sum_{\omega \in \Omega} H(\omega) \log(H(\omega)/B(\omega))$ of $H$ from $B$ (also known as the \emph{Kullback-Leibler divergence}). Here $\log$ is the natural logarithm to base $e$. Note that $d_{\TV}$ is symmetric, but $D_{\KL}$ is not symmetric in general.

Stam \cite{Stam1978} gave the following upper bound for the relative entropy of the multivariate hypergeometric distribution from the multinomial approximation:
\begin{equation} \label{eq: StamBound}
D_{\KL}(H_{N, d, \chi} \parallel B_{d, \chi/N})
\le \frac{n d(d -1)}{2(N - 1)(N - d + 1)}.
\end{equation}
Applying \eqref{eq: StamBound} to \eqref{eq: Pinsker} with $\Omega = {\N^{n + 1}}_d$ yields:
\begin{align}
d_{\TV}(H_{N, d, \chi}, B_{d, \chi/N})
&\le \sqrt{\frac{1}{2} \cdot \frac{n d(d -1)}{2(N - 1)(N - d + 1)}} \notag\\
&< \frac{\sqrt{n}d}{2(N - d + 1)}, \text{whenever $d \ge 2$ and $n > 0$.} \label{eq: crude}
\end{align}
The authors acknowledge that the relaxation \eqref{eq: crude} is crude.

\section{A new bound of the P\'olya exponent} 
Fix $n \in \N$. The upper bound \eqref{eq: crude} of the total variation distance just obtained immediately yields an upper bound of the P\'olya exponent $\pi(P)$ (recall its definition in \eqref{eq: PolyaExponent}). Following Powers-Reznick \cite{PowersReznick2000}, introduce the following invariants for a form $P \in \R[x_0, \dots, x_n]$ of degree $d$:
\begin{align}
L(P) := \max_{\omega \in {\N^{n + 1}}_d} \frac{1}{\binom{d}{\omega}}\left| [x^\omega](P)\right| \quad \text{ and } \quad \lambda(P) := \inf_{x \in \Delta^n} P(x).
\end{align}

Note that $L(P) \ge |[{x_0}^d](P)| = P(1, 0, \dots, 0) \ge \lambda(P)$ for all $P$ strictly positive on $\Delta^n$.  

\begin{theorem} \label{thm: effectivePolya}
For every $(n + 1)$-ary form $P \in \R[x_0, \dots, x_n]$ of degree $d$ that is strictly positive on $\Delta^n$, 
\begin{equation} \label{eq: upperBoundOfPolyaExp}
\pi(P) \le \left\lfloor\sqrt{n} d \frac{L(P)}{\lambda(P)}\right\rfloor 
\end{equation} 
\end{theorem}

\begin{proof}[Proof of Theorem \ref{thm: effectivePolya}]
First, we dispose of some straightforward cases. 
If either $n = 0$ or $d < 2$, then every $n$-ary degree-$d$ form $P$ strictly positive on $\Delta^n$ already has strictly positive coefficients, thus $\pi(P) = 0$, so \eqref{eq: upperBoundOfPolyaExp} holds trivially.

Below, assume $n >0$ and $d \ge 2$. Let $N \ge d$. We follow the proof of Theorem \ref{thm: PolyaPositivstellensatz}. 
For each $\chi \in {\N^{n + 1}}_N$, the error of the approximation \eqref{eq: theApprox} is
\begin{align*}
\mathcal{E}_\chi &:= P(\chi/N) - \frac{1}{\binom{N}{\chi}}[x^\chi]((x_0 + \dots + x_n)^{N - d} P) \\
&= -\sum_{\omega \in {\N^{n + 1}}_d} \frac{1}{\binom{d}{\omega}} [x^\omega](P) \left( H_{N, d, \chi}(\omega) - B_{d, \chi/N}(\omega)\right).
\end{align*}
Taking the absolute value,
\begin{align*}
\left|\mathcal{E}_\chi\right|  &\le\sum_{\omega \in {\N^{n + 1}}_d} \frac{1}{\binom{d}{\omega}} \left|[x^\omega](P)\right| \left| H_{N, d, \chi}(\omega) - B_{d, \chi/N}(\omega)\right|  \\
&\le L(P) \cdot 2 d_{\TV}(H_{N, d, \chi}, B_{d, \chi/N}) \\
&< L(P) \cdot\frac{\sqrt{n}d}{N - d + 1}, 
\end{align*}
using \eqref{eq: crude} in the  last line. Therefore
\begin{align*}
\frac{1}{\binom{N}{\chi}}[x^\chi]((x_0 + \dots + x_n)^{N - d} P) &= P(\chi/N) - \mathcal{E}_\chi  \\
&\ge P(\chi/N) - \left|\mathcal{E}_\chi\right| \\
&> \lambda(P) - L(P) \cdot \frac{\sqrt{n} d}{N - d + 1}.
\end{align*}

To guarantee that $(x_0 + \dots + x_n)^{N - d} P$ has strictly positive coefficients, i.e. $[x^\chi]((x_0 + \dots + x_n)^{N - d} P) > 0$ for all $\chi \in {\N^{n + 1}}_N$, it thus suffices to ensure that
\[
\lambda(P) - L(P) \cdot\frac{\sqrt{n}d}{N - d + 1} \ge 0,
\]
or equivalently, solving for $N - d$,
\begin{equation*}
N - d \ge \sqrt{n}d \frac{L(P)}{\lambda(P)} - 1.
\end{equation*}
Therefore \eqref{eq: upperBoundOfPolyaExp} follows since $\pi(P)$ is an integer.
\end{proof}

Theorem \ref{thm: effectivePolya} ought to be regarded as a bound for $\pi(P)$ with the dimension $n$ of the simplex as fixed, while $d \in \N$ is varying. From this perspective, \eqref{eq: upperBoundOfPolyaExp} says that
\begin{equation} \label{eq: asymptoticsOfPolyaExp}
\pi(P) \in O_n(d) \frac{L(P)}{\lambda(P)} \quad \text{ as $d \to \infty$},
\end{equation}
where $f(n , d) \in O_n(d)$ if there is some constant $C_n > 0$ (depending only on $n$) and some $d_0 \in \N$ such that $|f(n, d)| \le C_n\cdot d$ for all $d \ge d_0$. 

For $n > 0$ fixed, the linear dependence on $d$ in \eqref{eq: asymptoticsOfPolyaExp} is tight, as the following example shows:

\begin{example} \label{eg: WhyEffPolyIsSharp}
Fix $n > 0$. 
For each $d \ge 3$, consider the degree-$d$ form:
\begin{equation}
P_d := (x_0 + \dots + x_n)^d - 2d x_0 x_1^{d-1}
\end{equation}
Note that $L(P_d) = 1$. It suffices to show that $\lambda(P_d) = \Omega(1)$ and that $\pi(P_d) = \Omega(d)$.

First, we minimise $P_d$ over $\Delta^n$.
By the weighted arithmetic mean-geometric mean inequality, we have $\sqrt[d]{x_0(x_1/(d-1))^{d - 1}} \le (x_0 + (d - 1) \cdot (x_1/(d - 1)))/d = (x_0 + x_1)/d \le (x_0 + x_1 + \cdots + x_n)/d = 1/d$ since $x \in \Delta^n$. Taking $d$-th powers,
\begin{equation*}
\frac{x_0 {x_1}^{d - 1}}{(d - 1)^{d - 1}} \le \frac{1}{d^d}.
\end{equation*}
Therefore, on $\Delta^n$,
\begin{align*}
P_d(x) = 1 - 2d x_0 {x_1}^{d - 1} 
&\ge 1 - 2d \cdot \frac{(d - 1)^{d - 1}}{d^d} \\
&= 1 - 2 \left( 1 - \frac{1}{d}\right)^{d - 1} \\
&\ge 1 - 2\left( 1 - \frac{1}{3}\right)^{3 - 1} = \frac{1}{9},
\end{align*}
since $(1 - 1/d)^{d - 1}$ is  decreasing in $d \ge 3$. Therefore $\lambda(P_d) \ge 1/9$.

Finally, we show that $(x_0 + \cdots + x_n)^d P_d$ has a negative coefficient, which would imply that $\pi(P_d) \ge d + 1$. Explicitly, 
\begin{align*}
&[{x_0}^2 {x_1}^{2d - 2}]((x_0 + \cdots + x_n)^d P_d) \\
= {} & [{x_0}^2 {x_1}^{2d - 2}]((x_0 + \cdots + x_n)^{2d}) - 2d\cdot [{x_0}^2 {x_1}^{2d - 2}](x_0{x_1}^{d - 1} (x_0 + \cdots + x_n)^d) \\
= {} & [{x_0}^2 {x_1}^{2d - 2}]((x_0 + \cdots + x_n)^{2d}) - 2d\cdot [{x_0}^1 {x_1}^{d - 1}]((x_0 + \cdots + x_n)^d) \\
= {} & \binom{2d}{2, 2d - 2, 0, \dots, 0} - 2d \cdot \binom{d}{1, d - 1, 0, \dots, 0} \\
= {} & -d < 0.
\end{align*}

Now suppose that $C_n(d) > 0$ is a constant depending on $d$ (recall that $n > 0$ is fixed) such that 
\begin{equation*}
\pi(P) \le C_n(d) \cdot \frac{L(P)}{\lambda(P)} \quad \text{ for all $P > 0$ on $\Delta^n$ of degree $d$.}
\end{equation*}
Substituting $P = P_d$ gives 
\begin{equation*}
d + 1 \le \pi(P_d)\le C_n(d) \cdot \frac{L(P_d)}{\lambda(P_d)} \le C_n(d) \cdot \frac{1}{1/9} = 9 C_n(d).
\end{equation*}
Hence $C_n(d) \ge (d + 1)/9 > d/9$.   Therefore
\begin{equation*}
C_n(d) = \Omega_n(d). \tag*{\qed}
\end{equation*}

\end{example}

\paragraph{Comparison with previous bounds.} The best upper bound for the P\'olya exponent explicitly stated in the literature is due to de Klerk-Laurent-Parrilo \cite{deKlerkLaurentParrilo2006}. Their upper bound uses an invariant (possibly) smaller than $L(P)$, namely
\begin{equation*}
P^{(0)}_{\max} := \max_{\omega \in {\N^{n + 1}}_d} \frac{1}{\binom{d}{\omega}} [x^\omega](P).
\end{equation*}
By carefully studying Powers-Reznick's effectivisation of P\'olya's original argument \cite{PowersReznick2001}, de Klerk-Laurent-Parrilo showed that 
\begin{equation} \label{eq: KLP}
\pi(P) \le \left\lfloor\frac{d(d-1)}{2} \cdot \frac{P^{(0)}_{\max}}{\lambda(P)}\right\rfloor - d + 1.
\end{equation}
Surprisingly, this upper bound is independent of $n$, the dimension of the simplex.
Their actual statement as \cite[Theorem 1.1]{deKlerkLaurentParrilo2006} is slightly different as they use the weak positivity certificate that $(x_0 + \cdots + x_n)^M P$ has nonnegative coefficients. As $\pi(P)$ is defined in terms of $(x_0 + \cdots + x_n)^M P$ having strictly positive coefficients, we express their work as the inequality \eqref{eq: KLP} instead, which follows from their proof. 

In polynomial optimisation over $n$-simplices, the viewpoint is that the degree $d$ of the polynomial is fixed, while $n \to \infty$. In this section, we have taken the geometric viewpoint, where the space $\Delta^n$, or basically its dimension $n$, is fixed, while $d \to \infty$.   

We may formally combine de Klerk-Laurent-Parrilo's bound with our bound in Theorem \ref{thm: effectivePolya} to obtain:
\begin{corollary}
For every $(n + 1)$-ary form $P \in \R[x_0, \dots, x_n]$ of degree $d$ that is strictly positive on $\Delta^n$,
\begin{equation*} 
\pi(P) \le \left\lfloor
\min\{(d - 1)P^{(0)}_{\max}, 2\sqrt{n} L(P)\} \frac{d}{2\lambda(P)}
\right\rfloor. 
\end{equation*} 
\end{corollary}
Note $2\sqrt{n} L(P)$ is smaller than $(d - 1)P^{(0)}_{\max}$ if 
\begin{equation}
n < \bigg((d-1) \frac{P^{(0)}_{\max}}{2L(P)}\bigg)^2.
\end{equation}

\section{A new bound for the Bernstein degree}

Setting $n = 1$, we obtain an effectivisation of Bernstein's result mentioned at the beginning of \S \ref{sec: intro}. 
Let $p \in \R[x]$ have degree $d$. 
For any $N \ge d$, we may express $p = \sum_{j = 0}^N c_{j, N} b_{j, N}$ in terms of the degree-$N$ Bernstein basis $\{b_{j, N}\}_{j = 0}^N = \{\binom{N}{j} x^j (1 - x)^{N -j}\}_{j = 0}^N$. Write $[b_{j,N}](p) := c_{j ,N}$ for the $j$-th Bernstein coefficient of degree $N$ (for $j = 0, \dots, N$). Following Powers-Reznick \cite{PowersReznick2000}, the maximum absolute value of the Bernstein coefficients of degree $d$ is denoted by:   
\begin{equation*}
L(p) = \max_{j = 0, \dots, d} |[b_{j,d}](p)|.
\end{equation*}

The \emph{(weak) Bernstein degree} of $p$ is the least $N \ge d$ such that $[b_{j,N}](p) \ge 0$ for all $j = 0, \dots, N$: see DeVore-Lorentz \cite{DeVoreLorentz1993}. The Bernstein degree of $p$ is denoted by $r(p)$ in \cite{PowersReznick2000}. We also define $r_{> 0}(p)$ as the \emph{strict Bernstein degree}, i.e. the least $N \ge d$ with $[b_{j,N}](p) > 0$ for all $j = 0, \dots, N$. Note
\begin{equation*}
r_{> 0}(p)
= \pi(P) + d,
\end{equation*}
where $P = \sum_{j = 0}^d c_{j, d} \binom{d}{j} x^j y^{d -j}$ is the homogenisation of $p = \sum_{j = 0}^d c_{j, d} b_{j, d}$. 

Setting $n = 1$ in Theorem \ref{thm: effectivePolya} and dehomogenising immediately gives:

\begin{corollary}
For every $p \in \R[x]$ of degree $d$ that is strictly positive on the unit interval $[0, 1] \subset \R$,   
\begin{equation} \label{eq: bernsteinAsym}
r(p) \le r_{> 0}(p) \le \left\lfloor d \frac{L(p)}{\inf_{[0, 1]} p}\right\rfloor + d.
\end{equation}
\end{corollary}
Asymptotically, \eqref{eq: bernsteinAsym} gives
\begin{equation} \label{eq: BernsteinAsmyptotics}
r_{> 0}(p) \in O(d) \frac{L(p)}{\inf_{[0, 1]}p} \quad \text{ as $d \to \infty$}.
\end{equation}
Naturally, setting $n = 1$ in de Klerk-Laurent-Parrilo's upper bound \eqref{eq: KLP} of the P\'olya exponent and dehomogenising also gives 
\begin{equation}
r_{> 0}(p) \in 
O(d^2) \frac{\max_{j =0, \dots, d} [b_{j,d}](p)}{\inf_{[0,1]}p}.
\end{equation}
Upper bounds of the Bernstein degree of univariate polynomials positive on $[0, 1]$ of arbitrary degree have been studied since Powers-Reznick \cite{PowersReznick2000}. 

The following one-parameter family of degree-$d$ univariate polynomials shows that the linear dependence on $d$ in \eqref{eq: BernsteinAsmyptotics} is tight. 

\begin{example} \label{eg: BernsteinCounter}
Fix $d \ge 1$. For $0 < \varepsilon < 1$, consider the following polynomial of degree $d$: 
\begin{equation*}
p_{\varepsilon} := 1 + \varepsilon x^d = \sum_{j = 0}^d (1 + \delta_j^d \varepsilon) b_{j,d}, 
\end{equation*}
where $\delta_j^d$ is the Kronecker Delta which is equal to $1$ when $j = d$ and otherwise is equal to $0$. Already, the degree-$d$ Bernstein coefficients of $p_\varepsilon$ are strictly positive, so $r_{>0}(p) = d$. Note also that $\lambda(p_\varepsilon) = 1$ and $L(p_{\varepsilon}) = 1 + \varepsilon$.

Now suppose that $C(d) > 0$ is a constant depending on $d$ such that 
\begin{equation}
r_{> 0}(p) \le C(d) \cdot \frac{L(p)}{\lambda(p)} \quad \text{ for all $p > 0$ on $[0, 1]$ of degree $d$.}
\end{equation}
Substituting $p = p_\varepsilon$ gives $d \le C(d) \cdot \frac{1 + \varepsilon}{1}$. Letting $\varepsilon \searrow 0$ gives $d \le C(d)$. Therefore
\begin{equation*}
C(d) = \Omega(d). \tag*{\qed}
\end{equation*}

\end{example}

\section{Concluding remarks}

Our exposition above has been deliberately thorough with an emphasis on conveying the intuition behind the estimates in this paper. We have avoided concision and hope that the reader learns more about this probabilistic interpretation of P\'olya's Positivstellensatz due to de Klerk-Laurent-Sun. In this AI era, merely rushing to define all terminology and then stating the main result does not benefit the reader any more than compiling a formal proof, for a sufficiently advanced AI would be able to one-shot any main results stated in the introduction section.

Apology aside, the reader not familiar with this subtopic in computational real algebraic geometry might wonder why our new bound in Theorem \ref{thm: effectivePolya}, different from that of de Klerk-Laurent-Parrilo in \eqref{eq: KLP}, is possible. Essentially, the multinomial approximation of the multivariate hypergeometric distribution is the estimate:
\begin{equation} \label{eq: Bernstein estimator}
\frac{1}{\binom{N}{\chi}}[x^\chi]((x_0 + \dots + x_n)^{N - d} P) 
 \approx P(\chi/N), 
\end{equation}
while P\'olya's estimate was 
\begin{equation} \label{eq: Polya estimator}
\frac{1}{\binom{N}{\chi}}[x^\chi]((x_0 + \dots + x_n)^{N - d} P) 
 \approx \frac{N^d}{N(N - 1) \cdots (N - d+ 1)} P(\chi/N). 
\end{equation}
Note, however that $N^d \sim N(N-1) \cdots (N - d + 1)$ as $N \to \infty$, for $d$ fixed, thus the estimators in the righthand sides of \eqref{eq: Bernstein estimator} and \eqref{eq: Polya estimator} are asymptotically equivalent. de Klerk-Laurent-Parrilo's effectivisation of P\'olya's estimate in \eqref{eq: KLP} leads to their upper bound \eqref{eq: KLP} of the P\'olya exponent, while our effectivisation of \eqref{eq: Bernstein estimator} leads to our alternative upper bound in Theorem \ref{thm: effectivePolya}.

A final historical remark is in order. When $n = 1$, the dehomogenisation of \eqref{eq: Bernstein estimator} was in fact the estimate Bernstein himself used in \cite{Bernstein1915}. Thus this paper comes full circle.

\paragraph{Open problems.} Two problems present immediately. First, what is the smallest constant $C > 0$ such that $r(p) \le C d L(p)/\inf_{[0, 1]} p$ for all $p \in \R[x]$ strictly positive on $[0, 1]$ (c.f.\ \eqref{eq: BernsteinAsmyptotics})? Example \ref{eg: BernsteinCounter} and Theorem \ref{thm: effectivePolya} constrain $1 \le C \le 2$ for the best $C$. 

Second, for fixed $n$, Theorem \ref{thm: effectivePolya} shows that $\pi(P) \in n^{1/2} O(d) L(P)/\lambda(P)$ for all $(n + 1)$-ary forms $P$ strictly positive on $\Delta^n$. Can the exponent $1/2$ of $n$ be lowered?

\paragraph{Acknowledgements and AI declaration.} 

We thank 
Dr Lampros Gavalakis,
Professor Adrian Roellin, 
Professor Rongfeng Sun,
Professor Wing-Keung To for helpful discussions and clarifications on matters in probability theory. 
The second author would also like to thank Ze Kang Tan, his former student, who worked on the effective P\'olya's Positivstellensatz over the $1$-simplex \cite{TanZeKang2018}. His work led us to conjecture the $O(d)$ upper bound.

We used ChatGPT-6 Astra during the research process.


\end{document}